\documentclass[11pt]{amsart}

\usepackage{amsthm, amsmath, amscd, amssymb, latexsym, stmaryrd, color}

\usepackage[all]{xypic}
\usepackage{color}
\usepackage[dvipsnames]{xcolor}
\usepackage{fullpage}
\definecolor{hot}{RGB}{65,105,225}
\usepackage[pagebackref=true,colorlinks=true, linkcolor=BurntOrange,  citecolor=Bittersweet, urlcolor=hot]{hyperref}

\usepackage[%
    left=1.in,%
    right=1.in,%
    top=1.5in,%
    bottom=1.5in,%
    paperheight=11in,%
    paperwidth=8.5in%
]{geometry}
\usepackage[all,arc]{xy}
\usepackage{enumerate}

\usepackage{mathtools}
\usepackage{graphicx}
\usepackage{hyperref}
\usepackage{parskip}
\usepackage{fancyhdr}

\usepackage[T1]{fontenc}
\usepackage{newtx}
\usepackage{tikz}
\usepackage{tikz-cd}
\usepackage{dsfont}
\usetikzlibrary{matrix}

\newtheorem{thm}{Theorem}[section]
\newtheorem{cor}[thm]{Corollary}
\newtheorem{prop}[thm]{Proposition}
\newtheorem{lem}[thm]{Lemma}
\newtheorem{conj}[thm]{Conjecture}
\newtheorem{quest}[thm]{Question}
\newtheorem{rem}[thm]{Remark}
\newtheorem{defn}[thm]{Definition}

\newtheorem{exmp}[thm]{Example}

\makeatletter
\let\c@equation\c@thm
\makeatother
\numberwithin{equation}{section}

\makeatletter
\def\thm@space@setup{%
  \thm@preskip=\parskip \thm@postskip=0pt
}
\makeatother

\title{Topological Obstructions to Dynamical Convexity}
\date{}
\author{Shahnaz Shamim Shahul}
\address{Université de Toulouse\\CNRS\\Institut de Mathématiques de Toulouse \\ France\\ }
\email{shahnaz\_shamim.shahul@math.univ-toulouse.fr}

\begin{document}
\maketitle

\begin{abstract}
  We study the topological obstructions of dynamical convexity on contact manifolds focusing on fillability by cotangent bundles and subcritical surgeries. Using links to algebraic geometry, we motivate and define a stronger version of dynamical convexity, and investigate the topology of these manifolds. More precisely, we show that  strongly dynamically convex contact manifolds cannot arise as a unit cotangent bundle $(ST^*M,\lambda_{std})$ of a closed  manifold $M$ and in particular that simply connected dynamically convex contact manifolds cannot be filled by  cotangent bundles. We demonstrate that dynamical convexity can  be used to recover homotopy groups of topologically simple fillings with vanishing symplectic homology.
  We also show obstructions to dynamical convexity  that come from studying different kinds of subcritical surgeries.
\end{abstract}

\section{Introduction}
One of the main themes in symplectic topology and symplectic dynamics is the relation between the topology of a symplectic manifold and the dynamics on its boundary; particularly when this boundary is a contact manifold. In this paper, we study constraints on the topology of the boundary and the interior, coming from the dynamics.

A key notion for the dynamics on a contact boundary, introduced by Hofer, Wysocki and Zehnder to generalize the concept of affine convexity is \textbf{dynamical convexity}. In particular, they proved that the boundary of any affine convex domain in $\mathbb{R}^{2n}$ is dynamically convex. It was proven recently that the converse is not true \cite{chaidez20223d, dardennes2024symplectic, chaidez2025ruelle}.

The assumption of being dynamically convex is key in many results in dynamics, for instance, \cite{hofer1999characterization, gutt2016minimal, abreu2017multiplicity, ginzburg2021dynamical, hryniewicz2022global},and for existence of some invariants, for instance in \cite{hutchings2017cylindrical, zhou2021symplectic, chaidez2024contact}. We also highlight its relations to algebraic geometry, motivated by \cite{mclean2016reeb}. The question we study in this paper is
\begin{quest}
    Which manifolds admit a dynamically convex contact structure?
\end{quest}

A contact manifold $(M,\xi)$ is a smooth $2n-1$ dimensional manifold endowed with a maximally non-integrable hyperplane distribution. If we write $\xi$ as the kernel of a 1-form $\xi=\ker\alpha$, the condition is $\alpha\wedge{d\alpha}^{n-1}\neq0$. If $\alpha$ is globally defined, which we shall always assume, we call $\alpha$ a contact form.
Given  a contact form $\alpha$ on $(M,\xi)$, there exists a unique vector field on $M$ called Reeb vector field $R_\alpha$ defined by 
$$R_\alpha\in \ker(d\alpha) \quad \textrm{ and } \quad \alpha(R_\alpha)=1.$$

 Let $\phi^{\alpha}$ be the flow generated by the Reeb vector field. We denote the set of closed orbits of this flow, called {\bf Reeb orbits} by $\mathcal{P}(\alpha)$.
 Given a Reeb orbit $\gamma\in\mathcal{P}(\alpha)$, we can associate a number to $\gamma$ based on how much the flow of $R_\alpha$ winds around $\gamma$. This number is called the {\bf Conley Zehnder index} of $\gamma$, $\mu_{CZ}(\gamma)$. 
 More precisely, in a symplectic trivialization of $\gamma^\star TM$, the flow $\phi^{\alpha}$ gives a path of symplectic matrices to which we can associate a number; see \cite{conley1984morse, robbin1993maslov, gutt2012conley}.

 A priori, the Conley Zehnder index depends on the choice of trivialization. This is not the case if the Reeb orbit is contractible and $c_1(\xi)=0.$
We can also  have a well-defined Conley Zehnder index (independent of the trivialization) valued in $\mathbb{Q}$ denoted by $\mu^\mathbb{Q}_{CZ} $ if $c_1(\xi) $ is torsion. This latter index admits a lower semi-continuous extension $\mu^\mathbb{Q}_{LCZ}$ of $\mu^\mathbb{Q}_{CZ}$ \cite{mclean2016reeb}.
\begin{defn}
Given a $2n-1 $ dimensional contact manifold $(M,\xi$) with $c_1(\xi,\mathbb{Q})=0$. A contact form $\alpha$  is dynamically convex if $\mu^\mathbb{Q}_{LCZ}(\gamma)\geq n+1$ for all $\gamma$ contractible Reeb orbits. A contact manifold($M,\xi$) is said to be dynamically convex if there exists a contact form $\alpha$ with $\xi = \ker\alpha$ which is dynamically convex.
\end{defn}

 It was proved in \cite{hofer1999characterization}, that if a contact 3-manifold admits a dynamically convex contact form, then the contact structure is tight and the second fundamental group of the manifold vanishes, thereby constraining its topology. Recently, \cite{kwon2024dynamically} proved that if a contact manifold is dynamically convex, and simply connected, then it is contactomorphic to the standard sphere, provided the manifold admits a flexible Weinstein filling. This gives a characterisation of standard contact spheres.
 
 A very natural question is, \emph{which contact manifolds can be equipped with a dynamically convex contact form and what are the topological obstructions to do so?}
 
 We prove that a large class of examples cannot admit such contact forms. Most notably, simply connected unit cotangent bundles of  manifolds cannot admit dynamically convex contact forms. 
 \begin{thm}[Theorem\ref{change}]\label{contra}
        The unit cotangent bundle $ST^*M$ of a closed manifold $M$ cannot be dynamically convex if it is simply connected .
\end{thm}

Dynamical convexity is a condition defined only on contractible Reeb orbits.
In \cite{mclean2016reeb}, Mclean showed that the Reeb dynamics on a the link of a singularity of an affine variety is related to an algebro-geometric invariant of the singularity called the minimal discrepancy. Using Mclean's results, a conjecture in birational geometry concerning this invariant naturally leads to a stronger version of dynamical convexity which we call {\bf strong dynamical convexity }(Definition \ref{sdcdef}). Motivated by these results we want to also  study the topology of strongly dynamically convex contact manifolds.

We recall some definitions and results from \cite{mclean2016reeb}.
Given an irreducible affine variety $V$ of complex dimension $n$ in $\mathbb{C}^N$ with an isolated singularity at $0$ (including the case where $V$ is smooth),  the intersection of $V$ with a sphere $S^{2N-1}(\epsilon)$ centered at $0$ for small enough $\epsilon$, is a $2n-1$ differentiable  manifold $L=V\cap S^{2N-1}(\epsilon)$. $L$ naturally inherits a contact structure $\xi_L=TL\cap i(TL)$.

The singularity of $V$ at $0$ is called smooth if $V$ is smooth at $0$ and topologically smooth if the link $L$ is diffeomorphic to $S^{2n-1}$. Note that if a singularity is smooth then it is topologically smooth. Mumford \cite{mumford1961topology} proved that in complex dimension $2$, a singularity is smooth if and only if it is topologically smooth. However, in complex dimension $3$ and higher there are examples of non-smooth singularities which are  topologically smooth. Therefore, the smoothness of the singularity can  be determined by the diffeomorphism type of the link $L$ only in dimension 2.

However, if we consider $L$ with its inherited contact structure $\xi_L$,  Mclean proved that in complex dimension $3$, a singularity of a normal variety $V$ is smooth if and only if the link $(L,\xi_L)$ is contactomorphic to $(S^{5},\xi_{std})$.

If the variety $V$ is normal and the singularity is numerically $\mathbb{Q}-$Gorenstein, then we can define an invariant of this singularity called minimal discrepancy denoted by $\operatorname{min.disc}(V,0)$. Mclean relates this invariant to a contact invariant of the link called $\operatorname{hmi}$ (highest minimal index).

\begin{defn}
    For a contact manifold $(M,\xi)$ with $c_1(\xi,\mathbb{Q})=0$ and  $H^1(M,\mathbb{Q})=0$, the highest minimal index is defined as
    \[  
    \operatorname{hmi}:= sup_{\{\alpha\,|\,\xi=ker\hspace{0.05cm} \alpha \}} inf_{\{\gamma\in \mathcal{P}(\alpha)  \} }\Bigl(\mu^{\mathbb{Q}}_{LCZ}(\gamma)+n-3\Bigr)
    \]
\end{defn}
\begin{thm}(\textbf{Mclean}\cite{mclean2016reeb})\label{mclean} Given $H^1(L,\mathbb{Q})=0$, and $V$, an affine normal variety with an isolated numerically $\mathbb{Q}-$Gorenstein singularity, then
\begin{enumerate}
\item if $\operatorname{min.disc}(V,0)\geq 0$, $2\operatorname{min.disc}(V,0)=\operatorname{hmi}(L)$
\item if $\operatorname{min.disc}(V,0)<0$, $\operatorname{hmi}(L)<0$
\end{enumerate}
\end{thm}

A conjecture that relates minimal discrepancy to the smoothness of the singularity was formulated by Shokurov and is open for varieties of complex dimension greater than 3.

\begin{conj}[Shokurov]\label{shok} For a normal and affine variety $V$ of complex dimension $n$ with an isolated singularity at $0$ which is  numerically $\mathbb{Q}$-Gorenstein, 
$$\operatorname{min.disc}(V,0)\leq n-1$$ with the singularity being smooth if and only if $\operatorname{min.disc}(V,0)=n-1$.
\end{conj}

For links of singularities as in Shokurov conjecture, assuming $\operatorname{min.disc}(V,0)>0$, we have, using \ref{mclean}
that 
$$sup_{\{\alpha:\xi=ker\hspace{0.05cm} \alpha \}} inf_{\{\gamma\in \mathcal{P}(\alpha)  \} }(\mu^{\mathbb{Q}}_{LCZ}(\gamma))\leq n+1.$$ Therefore, if conjecture \ref{shok} is true, there cannot exist a contact form for links such that 
$$inf_{\{\gamma\in \mathcal{P}(\alpha)\}}(\mu^{\mathbb{Q}}_{LCZ}(\gamma))>n+1.$$
Furthermore, $\operatorname{min.disc}=n-1$ is equivalent to the link $(L,\xi_L)$ admitting a contact form $\alpha$ such that $inf_{\{\gamma\in \mathcal{P}(\alpha)\}}(\mu^{\mathbb{Q}}_{LCZ}(\gamma))=n+1$.  This naturally lead us to the definition of strong dynamical convexity.

\begin{defn}\label{sdcdef}  Given a $2n-1 $ dimensional contact manifold $(M,\xi$) with $c_1(\xi)$ torsion and $H^1(M,\mathbb{Q})=0$. A contact form $\alpha$ is strongly dynamically convex if $$\mu^\mathbb{Q}_{LCZ}(\gamma)\geq n+1$$ for all Reeb orbits $\gamma$. ($M,\xi$) is said to be strongly dynamically convex if there exists a strongly dynamically convex contact form $\alpha$ with $\ker\alpha=\xi$.    \end{defn}

\begin{rem}
    This definition of strong dynamical convexity is different from the one introduced in \cite{ginzburg2021dynamical}.
\end{rem}
\begin{rem}
     In contrast with dynamical convexity, strong dynamical convexity is defined taking into account all Reeb orbits, not just contractible ones. The topological conditions placed on the manifold are precisely those that allow us to do so. Note that when $\pi_1(M)=0$, the notions of dynamical convexity and strong dynamical convexity coincide.
\end{rem}

Shokurov's conjecture, if true,  implies that $L$ is diffeomorphic to a sphere if $L$ is strongly dynamically convex and arises as a link of a normal numerically $\mathbb{Q}-$Gorenstein singularity. Furthermore, using $$sup_{\{\alpha:\xi=ker\hspace{0.05cm} \alpha \}} inf_{\{\gamma\in \mathcal{P}(\alpha)  \} }(\mu^{\mathbb{Q}}_{LCZ})(\gamma)\leq n+1 $$ Shokurov's conjecture also implies that such links cannot support a contact form with $$inf_{\{\gamma\in \mathcal{P}(\alpha)\}}\mu^\mathbb{Q}_{LCZ}(\gamma)\geq n+2.$$
It is natural to ask whether these questions hold more generally?

\begin{quest}\label{m}
    If a contact manifold $(M,\xi)$ is strongly dynamically convex, then is it diffeomorphic to $S^{2n-1}$? 
\end{quest}
\begin{quest}\label{1.7}
     Does there exist a contact manifold $(M,\xi)$ supporting  a contact form $\alpha$, such that $inf_{\{\gamma\in \mathcal{P}(\alpha)\}}\mu^\mathbb{Q}_{LCZ}(\gamma)\geq n+2$?   
    \end{quest}

\begin{rem}
    An immediate observation in the direction of \ref{1.7} is that if  a contact manifold admits a  contact form $\alpha$, such that $inf_{\{\gamma\in \mathcal{P}(\alpha)\}}\mu^\mathbb{Q}_{LCZ}(\gamma)\geq n+2$, then it cannot admit fillings $W$ with vanishing symplectic homology, in particular flexible fillings. This follows from the long exact sequence in symplectic homology \ref{les} and the fact that $SH^+_{n+1}(W)=0$ which would force $H^0(W)=0$, a contradiction.  
\end{rem}
\begin{rem}
    Conjecture \ref{shok} is proven for toric varieties and up to complex  dim $3$. Mclean uses \ref{mclean} and the proof of Shokurov's conjecture in complex dim $3$ to show that a singularity is smooth if and only if it is contactomorphic to the standard sphere in complex dimension $3$. Therefore strongly dynamically convex contact manifolds that arise as links are not just diffeomorphic but contactomorphic to the standard sphere up to real dimension $5$. We can thus strengthen our Question \ref{m} to ask whether strongly dynamically convex manifolds are contactomorphic to standard spheres? Note that this is related to the problem of charaterizing $(S^{2n-1},\xi_{std})$. Since strong dynamical convexity and dynamical convexity coincide on a sphere, this would imply that the assumption of flexible fillability can be removed from the assumptions of \cite[theorem A]{kwon2024dynamically}.
\end{rem}

Question \ref{1.7} would imply that strong dynamical convexity is the ``limit case'' for highest minimal index. A proof of \ref{m} for links of singularities would thus prove a strong  implication of conjecture \ref{shok}. 

Our main result is about strongly dynamically convex unit cotangent bundles.
\begin{thm}[Corollary\ref{sdc}] \label{thm 1.11}
    The unit cotangent bundle $ST^*M$ of a closed manifold $M$ cannot be strongly dynamically convex.
    \end{thm}
    \begin{rem}
        A consequence of theorem \ref{thm 1.11} is that simply connected unit cotangent bundles cannot be dynamically convex. Note that unit cotangent bundles of simply connected manifolds are simply connected if $n\geq 3$ by using the long exact sequence of homotopy groups of a fibration. For $n=2$, the unit cotangent bundle of $S^2$ which is $\mathbb{RP}^3$ is dynamically convex, however not strongly dynamically convex. \end{rem}
    Another application of  Theorem \ref{thm 1.11} is to check conjecture \ref{shok} for varieties with links of singularities that are contactomorphic to unit-cotangent bundles. A toy example is given by the affine variety $ V(z_1^2+z_2^2+\cdots+z_n^2)$ with singularity at origin whose link is contactomorphic to $(ST^*S^{n-1},\xi_{std}).$ \cite{kwon2016brieskorn}. Other examples include unit cotangent bundles of manifolds considered in \cite[Theorem $1.1$]{mclean2018affine}.

In a positive step towards Question \ref{m}, we have
\begin{thm}
    If a simply connected contact  3 manifold $M$ is dynamically convex, then $M$ is contactomorphic to $(S^3,\xi_{std})$  
\end{thm}
\begin{proof}
    If M is dynamically convex, then by \cite[Theorem 1.5]{hofer1999characterization} $\xi$ is tight. By Poincar\'e's theorem $M$ is  diffeomorphic to a sphere $S^3$. The uniqueness of tight contact structure on $S^3$ \cite{eliashberg1992contact} implies $M$ is contactomorphic to the standard tight contact sphere $(S^3,\xi_{std}).$ 
\end{proof}
\begin{prop}\label{rational}
    If a contact $3$-manifold $M$ is strongly dynamically convex, then $M$ is a rational homology sphere
\end{prop}
\begin{proof}
    For a manifold to be strongly dynamically convex, we impose the condition $H^1(M,\mathbb{Q})=0$ in the definition for having a well defined Conley Zehnder index for non-contractible orbits. The universal coefficients theorem  implies  $H_1(M,\mathbb{Q})=0$ since $H^1(M,\mathbb{Q})=0$ and Poincar\'e duality  implies $H_2(M,\mathbb{Q})=0$. Therefore, $M$ is a rational homology sphere. 
\end{proof}

Note that there exists examples of  dynamically convex contact manifolds  which are not spheres but are quotients of spheres. We provide some of these examples in the next section.

\cite{kwon2024dynamically} proved that dynamically convex contact manifolds with fillings having vanishing symplectic homology are homology spheres and when $\pi_1=0$, homeomorphic to spheres. We prove that we can obtain homotopical information of such fillings without the condition on  $\pi_1$  by working with symplectic homology of covering spaces. More precisely, we prove the following statement.

\begin{thm}[Theorem \ref{hom}]\label{homotop}
      Let $M$ be a dynamically convex contact manifold such that there exists a topologically simple filling $W^\prime$ of $M$ with $SH(W^\prime)=0$. Then for any topologically simple filling $W^\prime$ of $M$, $\pi_k(W)\cong \pi_k(M)$ for $1<k<2n-1$ if $\pi_1(\partial W)\rightarrow \pi_1(W)$ is surjective.
\end{thm}

In another direction, we looked at building dynamically convex manifold using flexible/subcritical surgeries (see \cite{weinstein1991contact, cieliebak2002handle, fauck2016handle, lazarev2020contact}). This did not present new examples but obstructed large classes of manifolds from being dynamically convex, and hence strongly dynamically convex . 
\begin{rem}
    Throughout the paper, when we refer to surgery, we use the framework of contact surgery as in \cite{weinstein1991contact, fauck2016handle, geiges2008introduction} with the same framings.
\end{rem}

The following is a list of the obstructions we obtain.

 \begin{thm}[Theorem \ref{surg}]\label{subcritical}
      Let $M$ be a dynamically convex Weinstein fillable contact manifold of dimension greater than $5$. Let $M^\prime$ be the manifold that is obtained from $M$ by subcritical or flexible surgery, then $M^\prime$ cannot be dynamically convex. \end{thm}
      \begin{thm}[Theorem \ref{con 0}]\label{contact 0}
    Let $M^\prime $ be the manifold obtained by contact $0-$surgery on a Weinstein -fillable contact manifold $M$ with vanishing rational first chern class. Then $M^\prime$ cannot be dynamically convex.
\end{thm}
\begin{thm}[Corollary \ref{SH_n}]\label{SH}
    Let $ \Lambda$ be the class of manifolds that admit fillings with non-vanishing symplectic homology, then for $Y_1,\ldots Y_k\in\Lambda$, $k>1$, the connected sum $\#_kY_i$ and the manifolds obtained by subcritical/flexible surgeries on this class of connected sums cannot be dynamically convex. 
\end{thm}
\begin{thm}[Theorem \ref{hom1}]\label{hom}
    Let $W^{2n}$ be a Weinstein domain which is a rational homology ball , that is $H_k(W,\partial W)=\mathbb{Q}$ for $k=2n$ and $0$ otherwise. Assume $n\geq 3$, then attaching a flexible/subcritical handle produces a manifold $W^\prime$ such that $\partial W^\prime$ cannot be dynamically convex.
\end{thm}

      \begin{thm} [Theorem \ref{cot con}]\label{cotangent connect}
    The connected sum of any Weinstein fillable contact manifold and a unit cotangent bundle $DT^*N$ of a simply connected manifold $N$ cannot be dynamically convex 
\end{thm}   

\subsection*{Organization of the paper}
In  section \ref{2}, we recall the definitions of dynamical and strong dynamical convexity,providing examples and counterexamples of these manifolds. In section \ref{3}, we review symplectic homology and prove theorem \ref{hom} using symplectic homology with local coefficients. We then prove theorem \ref{thm 1.11} and \ref{contra} using  Serre spectral sequences and Viterbo's isomorphism in section \ref{cot}. In section \ref{5}, we prove the list of obstructions to dynamical convexity using subcritical surgery, more specifically theorems \ref{subcritical},\ref{contact 0},\ref{SH},\ref{hom},\ref{cotangent connect}.

\subsection*{Acknowledgements}
I am indebted to both of my advisors, Jean François Barraud and Jean Gutt  for all the discussions and advice  and for their meticulous help with the exposition of this article. I would like to thank Alberto Abbondandolo for valuable comments regarding the first version of the paper. I would also like to acknowledge the financial support provided by EUR-MINT at the Institut de Mathématiques de Toulouse.

\section{Dynamical Convexity} \label{2}

Let $(M,\xi)$ be a $2n-1$ dimensional contact manifold. Let $\phi^{\alpha}$ be the Reeb flow and recall that we denote the set of closed orbits of this flow on $M$ to be $\mathcal{P}(\alpha)$.

Given a path in $Sp(2n-2)$ that starts at $id$ and ends in a matrix $A\in Sp(2n-2)$ such that $det(A-id)\neq 0$, one can associate an integer called Conley Zehnder index $\mu_{CZ}.$(cf\cite{mclean2016reeb},\cite{gutt2012conley} for details). The Conley Zehnder index  $\mu_{CZ}$ has a lower semicontinous extension in $Sp(2n-2)$ called $\mu_{LCZ}$ cf \cite{mclean2016reeb}.

For a contact manifold $(M,\xi)$ and a Reeb orbit $\gamma$, a trivialization of $\gamma^*\xi$ gives a path in $Sp_{2n-2}$ to which we can associate $\mu_{LCZ}$. However $\mu_{LCZ}$ depends on this trivialization. Note that a trivialization of $det_\mathbb{C}\xi$ suffices to define Conley-Zehnder indices of all the orbits and we can show that the indices is independent of the trivialization for null-homologous or contractible orbits. The assumption $c_1(\xi)=0$ is needed to show that such a global trivialization exists.

In order to have a well defined canonical grading, we can assume that $c_1(\xi)$ is torsion , which suffices to define  rational Conley Zehnder index $\mu^\mathbb{Q}_{CZ}$  of  contractible Reeb orbits taking values in $\mathbb{Q}$. When $H^1(M,\mathbb{Q})=0$, we can obtain a well defined $\mu^\mathbb{Q}_{CZ}$ for not just contractible orbits but all orbits as the set of homotopy classes of trivializations are determined by $H^1(M,\mathbb{Q})$ . (see\cite[Lemma $4.3$]{mclean2016reeb},\cite{li2025kahler} for details) . 

\begin{rem} \label{coincide}The rational index $\mu^\mathbb{Q}_{LCZ}(\gamma)$ of a Reeb orbit $\gamma$ coincides with the integer index $\mu_{LCZ}(\gamma)$ when the first chern class $c_1(\xi)$ is 0.\end{rem}

We say that a Reeb orbit $\gamma$ of period $T$ of a contact form $\alpha$ is non-degenerate if 
$det(d{\phi^\alpha}_T-id)\neq 0 .$ We say that $\alpha$ is non-degenerate if all of its Reeb orbits are non-degenerate. In \cite{mclean2016reeb}, $\mu^\mathbb{Q}_{LCZ}$ of an orbit is defined to be the infimum of non-degenerate nearby pertubations. Therefore $\mu^\mathbb{Q}_{LCZ}(\gamma)$ coincides with $\mu^\mathbb{Q}_{CZ}(\gamma)$ when $\gamma$ is non-degenerate. Therefore we can equivalently define dynamical convexity by using $\mu^\mathbb{Q}_{CZ}$ instead of $\mu^\mathbb{Q}_{LCZ}$ by restricting to non-degenerate contact forms.

\begin{rem}
    The definition of dynamical convexity that is prevalent in literature, for instance in \cite{kwon2024dynamically}\cite{hofer1999characterization}, usually includes the condition that the first chern class $c_1(\xi)$=0 instead of $c_1(\xi,\mathbb{Q})=0$. Note however that by \ref{coincide}, we have that normal definition is the same as ours when $c_1(\xi)=0$.   
\end{rem}

We recall the definiton \ref{sdc} of strong dynamical convexity.

\begin{defn}\textit{Strong Dynamical Convexity}: Let $(M,\xi)$ be a $2n-1 $ dimensional contact manifold with $c_1(\xi)$ torsion and $H^1(M,\mathbb{Q})=0$ . A contact form $\alpha$ ($\xi=$ker $\alpha$) is strongly dynamically convex if $\mu^\mathbb{Q}_{LCZ}(\gamma)\geq n+1$ for $\gamma$ a Reeb orbit of $\alpha$. ($M,\xi$) is said to be strongly dynamically convex if $\xi$ supports a strongly dynamically convex contact form $\alpha$
        
\end{defn}

\begin{rem}
    Strong dynamically convexity by definition  implies dynamical convexity. Note that when $\pi_1(M)=0,$ this would imply by Hurewicz and universal coefficients theorem that $H^1(M,\mathbb{Q})=0.$ Therefore dynamical convexity and strong dynamical convexity coincide for simply connected manifolds.
    \end{rem}

The only example known to us of a strongly dynamically convex manifold is $(S^{n-1},\xi_{std})$  and we conjecture that strongly dynamically convex manifolds are diffeomorphic to spheres. However we can find examples of dynamically convex manifolds, which are not strongly dynamically convex.

\begin{prop}\label{quotient}
    Let $(M,\xi)$ be a contact manifold with a dynamically convex contact form $\alpha$, $c_1(\xi)=0$ and suppose a discrete group $G$ acts on $M$ such that the action is free and proper and $\alpha$ is $G-$invariant. Then $M/G$ with the induced contact form $\hat{\alpha}$ is dynamically convex if $c_1(ker\hat{\alpha},\mathbb{Q})=0$. 
    \end{prop}
   \begin{proof}
       We obtain a covering map $M\xrightarrow{p}M/G$ since the action is free and proper. Now a contractible orbit $\gamma$ in $M/G$ lifts to a contractible Reeb orbit $\hat{\gamma}$ in $M$ by using the lifting lemma with respect to a filling of $\gamma$ by a disk $D$ and the the fact that $p^*\hat{\alpha}=\alpha.$ Now around a neighbourhood of a lift of $D$, say $\hat{D}$, we have a strict contactomorphism induced by $p$, which would imply that the $\mu_{LCZ}(\hat{\gamma})=\mu_{LCZ}(\gamma)$. This would imply $M/G$ is dynamically convex as $\mu_{LCZ}(\hat{\gamma})\geq n+1$.
   \end{proof}
   \begin{exmp} \ref{quotient} would imply that lens spaces $L(p,q)$ with the standard contact structures they inherit from $S^3$ are dynamically convex. However, note that lens spaces arise as links of singularities\cite{michel2020topology}, therefore conjecture \ref{shok} which has been proved for links in dimension 3  implies that they are not strongly dynamically convex.        
   \end{exmp}
   \begin{exmp}
      Note that strongly dynamically convex $3-$ manifolds are necessarily rational homology spheres by \ref{rational}, thereby apart from lens spaces, a natural candidate to examine is the Poincare homology sphere $(P,\xi)$ with its standard contact structure. However $P$ being a brieskorn manifold arising as the link of a Brieskorn polynomial implies that its not strongly dynamically convex by the proof of \ref{shok} upto complex dimension $2$. 
   \end{exmp}
   \begin{defn}[\cite{bowden2022tight}]

       A contact manifold $(M,\xi)$ is said to be algebraically tight if its contact homology algebra $CH(Y,\Lambda)$ over Novikov field $\Lambda$ doesn't vanish.
   \end{defn}    
 \begin{prop}
     Let $(M,\xi)$ be a dynamically convex contact manifold. Then $(M,\xi)$ is (algebraically) tight.
 \end{prop}  
\begin{proof}
      \cite[Proposition $3.2$]{bowden2022tight}
    states that every contact manifold ($M^{2n-1},\xi$) with rational first chern class and supporting a contact form with no contractible orbit of conley zehnder index $4-n$ is algebraically tight. Algebraic tightness of a contact manifold implies tightness \cite{casals2019geometric} .
\end{proof}

\section{Symplectic Homology} \label{3}
In this section we recall basic facts about symplectic homology before using symplectic homology of covering spaces to prove theorem \ref{homotop}. We refer to the the following surveys and papers for a nice introduction with more details, \cite{cieliebak2018symplectic},\cite{lazarev2020contact},\cite{oancea2004survey},\cite{seidel2008biased}. We follow sign and grading conventions as in \cite{lazarev2020contact}, \cite{cieliebak2018symplectic}.

\begin{defn}
    A \textit{Liouville domain} $(W,\lambda)$ with a Liouville 1-form $\lambda$ is an exact symplectic manifold with symplectic form $d\lambda$ such that the Liouville vector field $X_\lambda$ defined by $i_{X_\lambda}(d\lambda)=\lambda$ that points transversely outward to $\partial W$
\end{defn}
\begin{defn}
    A Weinstein domain $(W,\lambda)$ is a Liouville domain that can support a Morse function $\varphi$ which has maximal level set $\partial W$  and such that $X_\lambda$ is a gradient-like vector field for $\varphi$.
\end{defn}

Note that the boundary $(\partial W,\lambda|_{}\partial W)$ of a Liouville domain $(W,\lambda)$ is a contact manifold  by definition with contact form $\alpha=\lambda|\partial W$ and we call $W$, the Liouville filling of $\partial W$. Given a Liouville domain $(W,\lambda)$, we can attach the positive part of the symplectization $[1,\infty)\times \partial W$ to $W$ at $\partial W$ to get the manifold $\hat{W}:=W\bigsqcup_{\partial W} [1,\infty)\times \partial W $ which is called the Liouville completion of $W$. We extend the Liouville form $\lambda$ on $W$ by $r\lambda|_{\partial W}$ on the symplectization $[1,\infty)\times \partial W$. Here $r$ denotes the cylindrical coordinate of the symplectization and the identification between the symplectization and the Liouville domain is made by the diffeomorphism induced by the Liouville flow near the boundary. 

We say that a Hamiltonian $H$ on $\hat{W}$ is admissible if $H$ is $C^2-$small in $W$ and $H=sr+c$ in $\hat{W}\textbackslash W$, where $s$ denotes the slope and $c\in \mathbb{R}$. We denote the set of admissible Hamiltonians by $H_{adm}$.
We also specify a class of admissible almost complex structures $J_{adm}$ on $\hat{W}$ such that for $J\in J_{adm}$, we have $Jker\alpha=ker\alpha$, $J|ker\alpha$  is a compatible almost complex structure on the bundle $\ker \alpha$ and independent of $r$, $J(r\partial_r)=R_\alpha$ on $\hat{W} \textbackslash W $ and $J$ is compatible with the symplectic form on $\hat{W}$.

 For our choice of $H\in H_{adm}$, the periodic trajectories $\mathcal{P}(H)$of the Hamiltonian vector field $X_H$ correspond to Morse critical points on $W$ and Reeb orbits on a level $r$ in $\hat{W}/W$ of period $s$ since $X_H=sR_\alpha.$ The Hamiltonian orbits are degenerate at the cylindrical end and we  consider a small time-dependent pertubation of $H$, so that each $S^1$ family degenerates into two orbits\cite{bourgeois2009exact}.

 Given $H\in H_{adm}$ and $J\in J_{adm}$, we define $SC(W,H,J):=\bigoplus _{\gamma\in \mathcal{P}(H)}{\mathbb{Z}}$ where the Hamiltonian is considered after a small-time dependent pertubation. When $c_1(W,\mathbb{Q})=0,$ we can associate a rational grading $\mu^\mathbb{Q}_{CZ}$ to orbits in $\mathcal{P}(H)$ and therefore to $SC(W,H,J)$. We can define a differential $\partial:SC_{k}\rightarrow SC_{k-1}$ such that $\partial \gamma=\bigoplus_{x\in SC_{k-1}}|\mathcal{M}(x,\gamma,H,J)/\mathbb{R}|x$.  After a small-time dependent pertubation of $(H,J)$, $\mathcal{M}(x,\gamma,H,J)$ can be shown to be a compact $1-$dimenisonal manifold , where $\mathcal{M}(x_-,x_+,H,J):=\{u:\mathbb{R}\times S^1\rightarrow \hat{W};\lim_{s\rightarrow\pm\infty}u(s,.)=x_\pm,\partial_su+J\partial_tu-JX_H=0\}$. After quotienting with $\mathbb{R}$ and then considering $0-$dimensional moduli spaces we can show $\partial\circ\partial=0$ and obtain the homology group of the chain complex which we denote by $SH(W,H,J)$.

 When slope of $H_1$ is less than the slope of $H_2$, there exists a continuation map $SH(W,H_1,J_1)\rightarrow SH(W,H_2,J_2)$ and these maps form a directed system. We define the \textbf{symplectic homology} $SH(W)$ of W to be direct limit $\underset{slope(H)}{\varinjlim} SH(W,H,J)$.

Note that the main difference between symplectic homology and Hamiltonian Floer homology is the fact that  the underlying manifold is non-compact. Hence we need to use a maximum principle to ensure the compactness of moduli spaces. However for the maximum principle to work in the moduli spaces involved in the continuation map, we need the Hamiltonians to have non-decreasing slope. 

We have a canonical filtration on symplectic homology induced by the action functional $\mathcal{A}_H$ for a small-time dependent perturbation of $H\in H_{adm}.$ Given $\gamma \in\mathcal{L}W$,  the action functional is defined by $\mathcal{A}_H(\gamma)=\int_{S^1}\gamma^*\lambda-\int_{S^1}H(\gamma(t))dt$. Note that, following the convention  in \cite{lazarev2020contact}, we consider the positive gradient flow in Floer trajectories which increases the action and therefore the differential decreases the action. This  implies the existence of  a well-defined filtered subcomplex $SC^{<k}(W,H,J)$ of $SC(W,H,J)$ whose homology is denoted by  $SH^{<k}(W,H,J)$. Taking direct limits we can define the filtered symplectic homology groups $SH(W)^{<k}:=\underset{slope(H)}{\varinjlim} SH(W,H,J)^{<k}$.

For a small time dependent perturbation of $H$ and $\epsilon>0$ small enough, we have that all the Hamiltonian orbits of $\hat{W}$ correspond to the Morse critical points in $W$ as $H$ is $C^2-$small. That is $SC^{<\epsilon}(W,H,J)$ is the Morse complex with a grading shift. Now the subcomplex $SC^{+,<\epsilon}:=SC/SC^{<\epsilon}$  is generated by Hamiltonian orbits corresponding to the Reeb orbits in the symplectization. We obtain an invariant  by taking the direct limit of the homology $SH^{+,<\epsilon}$ of this chain complex which is called \textbf{positive symplectic homology} $SH^+$.
\begin{equation}   
SH^+(W)=\underset{slope(H)}{\varinjlim} SH(W,H,J)^{+,<k}
\end{equation}

By construction the short exact sequence of chain complexes $0\rightarrow SC^{<\epsilon}\rightarrow SC\rightarrow SC^+\rightarrow 0$ gives an exact triangle between the Morse homology, symplectic homology and the positive symplectic homology.

\begin{rem}\label{contact form}
    The symplectic homology groups or positive symplectic homology groups do not depend on the contact form of $\partial W$ (cf \cite[section 3.9]{ritter2010deformations} ). More precisely, for a contact form $\alpha$ on $\partial W$, one can show that ${W}$ is Liouville isomorphic to a Liouville domain $W^\prime$ which induces the contact form $\alpha$ on $\partial W^\prime$ upto a strict contactomorphism.   
\end{rem}

\begin{prop}\label{les}
    For a Liouville domain $W^{2n}$, there exists a long exact sequence of groups
   
\[\begin{tikzcd}[cramped]
	{...} & {H_{*+n}(W,\partial W)} & {SH_{*}(W)} & {SH_{*}^{+}(W)} & {H_{*+n-1}(W,\partial W)} & {...}
	\arrow[from=1-1, to=1-2]
	\arrow[from=1-2, to=1-3]
	\arrow[from=1-3, to=1-4]
	\arrow[from=1-4, to=1-5]
	\arrow[from=1-5, to=1-6]
\end{tikzcd}\]
    
\end{prop}

\begin{rem}
    The grading shift occurs because the morse index of a critical point is related to the Conley Zehnder index by a difference of $n$.
\end{rem}

\begin{rem}
    The positive symplectic homology groups although being generated by Reeb orbits are not in general  invariants of the contact boundary. It depends on the filling as the Floer trajectories can enter the filling. However it can be shown that for contact manifolds that admit dynamically convex or more generally $ADC-$contact forms as in \cite{zhou2020vanishing}, the positive symplectic homology is independent of fillings $W$ such that $\pi_1(\partial W)\rightarrow \pi_1(W)$ is injective and $c_1(W,\mathbb{Q})$ is torsion. These fillings are called topologically simple.
\end{rem}
\begin{rem}\label{orb}
    For a strongly dynamically convex contact manifold $M^{2n-1}$, $SH_{k}^+(M)=0$ for $k< n+1$ since there are no generators in these groups. 
\end{rem}

Note that for a Liouville domain $W$ with contact boundary $(\partial W,\xi)$, $c_1(\xi,\mathbb{Q})=i^*c_1(TW,\mathbb{Q})$. Therefore $c_1(TW,\mathbb{Q})=0$ would imply $c_1(\xi,\mathbb{Q})=0$. For Weinstein domains of dimension greater than $6$, the converse is also true and we replicate the argument in \cite{lazarev2020contact} with rational coefficients. 

\begin{prop}\label{chern} \
    For a Weinstein domain $W$ of dimension $2n$ and contact boundary $(\partial W,\xi)$ with $n\geq 3$, we have $c_1(\xi,\mathbb{Q})=0$ iff $c_1(TW,\mathbb{Q})=0$.
\end{prop}
\begin{proof}
    For the inclusion $i:\partial W\rightarrow W$ we have $c_1(\xi,\mathbb{Q})=i^*c_1(TW,\mathbb{Q})$. This shows that $c_1(TW,\mathbb{Q})=0$ implies $c_1(\xi,\mathbb{Q})=0$. To prove the converse, note that any Weinstein domain $W^{2n}$ has the homotopy type of a CW complex of dimension at most $n$.Therefore $H_{2n-2}(W,\mathbb{Q})=0$ if $2n-2>n$. Now by Poincare-Lefshetz duality, we have $H^2(W,\partial W,\mathbb{Q})\cong H_{2n-2}(W,\mathbb{Q})=0$ when $n\geq 3$. Therefore the long exact sequence in cohomology 
       $H^2(W,\partial W,\mathbb{Q})\rightarrow H^2(W,\mathbb{Q})\xrightarrow{i^*} H^2(\partial W,\mathbb{Q})$ implies that $i^*$ is injective and therefore $c_1(TW,\mathbb{Q})=0$ if $c_1(\xi,\mathbb{Q})=0.$
    \end{proof}

\begin{rem} \label{rem}
   Note that the vanishing of positive symplectic homology is an invariant of dynamically convex contact manifolds when the manifold is simply connected. This is a consequence of the fact that a manifold that is simply connected and dynamically convex is by definition strongly ADC  and therefore linearized contact homology is well defined for the contact manifold (it has a unique augmentation) \cite{chaidez2024contact}, which vanishes if and only if the positive symplectic homology vanishes for fillable manifolds due to the existence of the Gysin sequence after identifying linearized contact homology with positive $S^1-$equivariant symplectic homology.
\end{rem}

\begin{prop}[\cite{viterbo1999functors}]\label{vit}
    Given a Liouville subdomain $W^\prime$ of a Liouville domain $W$, there exists a map called the Viterbo transfer map
    $SH(W)\rightarrow SH(W^\prime)$
\end{prop}

The above maps are used to show that attaching a subcritical handle induces an isomorphism in $SH$. For more details see \cite{fauck2016handle} or \cite{cieliebak2002handle}.

For cotangent bundles $T^*M$, Viterbo showed that there is an isomorphism between the $SH(T^*M)$ and the homology of the loop space $H_*(\mathcal{L}M)$ with $ \mathbb{Z}/2\mathbb{Z}$ coefficients. In fact this isomorphism can be lifted to $\mathbb{Z}$ coefficients after twisting by a suitable local system(\cite{abouzaid2013symplectic}). However, the isomorphism with $\mathbb{Z}/2\mathbb{Z}$ suffices for us. Note that for a cotangent bundle $T^*M$ of a manifold $M$, we have $c_1(T^*M)=0$ and conventions are chosen so that the gradings match.

\begin{thm}[\cite{viterbo1999functors}] \label{Viterbo}
\textbf{(Viterbo)}
    Given $T^*M$ the cotangent bundle of $M$, we have $SH_k(T^*M,\mathbb{Z}/2\mathbb{Z})\cong H_k(\mathcal{L}M,\mathbb{Z}/2\mathbb{Z}).$
\end{thm}

\begin{cor} \textbf{(Viterbo)}\label{pvit}
    For $T^*M$ , we have $SH^+_k(T^*M,\mathbb{Z}/2\mathbb{Z})\cong H_k(\mathcal{L}M,M,\mathbb{Z}/2\mathbb{Z}).$
\end{cor}

\begin{rem}
    Note that it is common to consider only contractible orbits in the definition of symplectic homology. Since the differential preserves the homotopy class of orbits, we get a decomposition $SH=\bigoplus_{\gamma\in \pi_1} SH_\gamma$ where $SH_\gamma$ denotes the subgroup generated by orbits in the homotopy class of $\gamma$. Proposition \ref{les} would go through if we restrict to contractible orbits as all the Morse critical points are included in this class of orbits. Furthermore this would imply  $SH(W)_\gamma\cong SH(W)^{+}_\gamma$ for $\gamma\neq 0$ in $\pi_1(W).$ \end{rem}

\subsection{Symplectic homology of covering spaces.}
Zhou \cite[Section 3.3]{zhou2021symplectic}\cite{zhou2023fillings}  defines and uses symplectic homology of covering spaces to study fillings of asymptotically dynamically convex contact manifolds. Note from here onward till the end of the section, we use symplectic homology generated with contractible Reeb orbits. Given a covering $\Tilde{W}\xrightarrow{\pi}W$, $SH(\Tilde{W})$ is defined as the group generated by lifts of the periodic orbits of $W$ and the differential is determined by counting the moduli spaces of (unique) lifted trajectories. More precisely the chain group is defined as $SC_*(\Tilde{W}):=\bigoplus_x \Pi_{\pi(\Tilde{x})=x}\Tilde{x}$ and it could be an infinite product in the fiber direction. The canonical map $x\rightarrow \Pi_{\pi(\Tilde{x})=x}\Tilde{x} $ induces maps on the level of $SH^{<\epsilon},SH, SH^+$ and fits into the long exact sequence of symplectic homology \cite[4.1]{zhou2023fillings}. Furthermore the map is a unital ring map in $SH$.

  \[\begin{tikzcd}
	{...} & {H^*(W)} & {SH_{*-n}(W)} & {SH^+_{*-n}(W)} & {...} \\
	{...} & {H^*(\tilde{W})} & {SH_{*-n}(\tilde{W})} & {SH^+_{*-n}(\tilde{W})} & {...}
	\arrow[from=1-1, to=1-2]
	\arrow[from=1-2, to=1-3]
	\arrow[from=1-2, to=2-2]
	\arrow[from=1-3, to=1-4]
	\arrow[from=1-3, to=2-3]
	\arrow[from=1-4, to=1-5]
	\arrow[from=1-4, to=2-4]
	\arrow[from=2-1, to=2-2]
	\arrow[from=2-2, to=2-3]
	\arrow[from=2-3, to=2-4]
	\arrow[from=2-4, to=2-5]
\end{tikzcd}\]

\begin{rem}
Zhou\cite[Section 3.3]{zhou2021symplectic}\cite{zhou2023fillings} uses cohomological convention and defines symplectic cohomology groups. They coincide with our definition if we define our grading using $-\mu_{CZ}$ instead of $\mu_{CZ}$
   \end{rem}

Zhou \cite[Remark 4.1]{zhou2023fillings} also discusses symplectic homology twisted by $\mathbb{Z}[\pi_1]$ coefficients and makes a comparison with the symplectic homology of covering spaces. In this case, the chain group is the $\mathbb{Z}[\pi_1]$ module generated by the periodic orbits with a twisted differential by the holonomy action of the fundamental group. When $\pi_1$ is finite, symplectic homology of the universal covering space coincides with the symplectic homology twisted with $\mathbb{Z}[\pi_1]$ coefficients since the covering space is compact and $SH(\tilde{W})$ is generated by finitely many orbits. When $\pi_1$ is infinite, the groups are different since symplectic homology of the universal covering space has infinitely many generators in a fiber direction.

\begin{rem}
    Singular cohomology of $W$ with $\mathbb{Z}[\pi_1]$ local coefficients is isomorphic to the compactly supported cohomology of the universal covering space $H^c(\tilde{W})$ and not $H(\tilde{W})$. However the cup product makes $H^c(\tilde{W})$, a module over  $H(\tilde{W})$.    
\end{rem}

\begin{thm}\label{module}
    Given a Liouville domain $W$, the pair of pants product makes the symplectic homology of $W$, twisted by $\mathbb{Z}[\pi_1]$ local coefficients a module over $SH(\tilde{W})$, the symplectic homology of the universal covering $\tilde{W}$.
\end{thm}
\begin{proof}
    We notice that the chain complex of symplectic homology with $\mathbb{Z}[\pi_1]$ local coefficients is a  subcomplex of that of symplectic homology of the universal cover $SH(\tilde{W})$ since its differential coincides with that of $SH(W,\mathbb{Z}[\pi_1])$ when restricted to finitely many orbits. Now let $z$ be the pair of pants product with an element $y$ represented by an infinite sum of orbits in the fiber direction and a single orbit $x$ . The projection to $W$ of the all the possible pair of pants between $x$ and summands in $y$ gives finitely many pair of pants in $W$ who have finite lifts after fixing $x$ due to unique lifting property as we are considering contractible orbits. Therefore $z\in SC(W,\mathbb{Z}[\pi_1])$ and we have a module action $SH(\tilde{W})\otimes SH(W,\mathbb{Z}[\pi_1])\rightarrow SH(W,\mathbb{Z}[\pi_1])$.
\end{proof}

\begin{cor}\label{loccoeff}
    Given a Liouville domain $W$, $SH(W)=0$ implies $SH(W,\mathbb{Z[\pi_1]}=0.$
\end{cor}
\begin{proof}
    Let $\tilde{W}$ be the universal covering of $W$. Since the canonical map $SH(W)\rightarrow SH(\tilde{W})$ is a unital ring map, $SH(W)=0$ implies $SH(\tilde{W})=0$. Theorem \ref{module} implies that $SH(W,\mathbb{Z}[\pi_1])=0$ since the unit of $SH(\tilde{W})$ acts as the identity. 
    \end{proof}

\begin{thm}\label{hom}\ref{homotop}
    Let $M$ be a dynamically convex contact manifold such that there exists a topologically simple filling $W^\prime$ of $M$ with $SH(W^\prime)=0$. Then for any topologically simple filling $W^\prime$ of $M$, $\pi_k(W)\cong \pi_k(M)$ for $1<k<2n-1$ if $\pi_1(\partial W)\rightarrow \pi_1(W)$ is surjective.
\end{thm}

  \begin{proof}
      We remark that by \cite{zhou2020vanishing}, the vanishing of symplectic homology for any topologically simple filling implies the vanishing of symplectic homology for all topologically simple fillings. Therefore, $SH(W)=0$. Let $\tilde{W}$ denote the universal covering of $W$. When $\pi_1(\partial W)\rightarrow \pi_1(W)$ is surjective, the boundary $\partial \tilde{W}$ is connected. Examining the long exact sequence of the relative homotopy groups of the pair $(\partial \tilde{W}, \tilde{W}) $, $\rightarrow \pi_1(\tilde{W})\rightarrow\pi_1(\tilde{W},\partial\tilde{W})\rightarrow\pi_0(\partial\tilde{W})\rightarrow \pi_0(\partial W) $, we deduce that $\pi_1(\tilde{W},\partial\tilde{W})=0.$ 

      Note that since $M$ is dynamically convex, $SH_k^+(W,\mathbb{Z}[\pi_1])=0$ for $k\leq n$ since there are no generators. After using Lefshetz duality to identify the compactly supported cohomology ${H^k}^c(\tilde{W})$ and $H_{2n-k}(\tilde{W},\partial \tilde{W})$, the canonical long exact sequence with $\mathbb{Z}[\pi_1]$ coefficients is given by
\[\begin{tikzcd}[cramped]
	{...} & {H_{*+n}(\tilde{W},\partial \tilde{W})} & {SH_{*}(W,\mathbb{Z}[\pi_1])} & {SH_{*}^{+}(W,\mathbb{Z}[\pi_1])} & {H_{*+n-1}(\tilde{W},\partial \tilde{W})} & {...}
	\arrow[from=1-1, to=1-2]
	\arrow[from=1-2, to=1-3]
	\arrow[from=1-3, to=1-4]
	\arrow[from=1-4, to=1-5]
	\arrow[from=1-5, to=1-6]
\end{tikzcd}\]
$SH(W)=0$ implies $SH(W,\mathbb{Z}[\pi_1])=0$ by corollary \ref{loccoeff}. Therefore the long exact sequence implies ${H_{k}(\tilde{W},\partial \tilde{W})}=0$ for $k<2n$. Relative Hurewicz theorem implies ${\pi_{k}(\tilde{W},\partial \tilde{W})}=0$ for $k<2n$. This implies $\pi_k(\tilde{W})\cong \pi_k(\partial \tilde{W})$ for  $k<2n-1$ using the long exact sequence of relative homotopy groups. Since projection from a covering space induces isomorphisms on $\pi_k$ for $k\geq 2$, we get $\pi_k(W)\cong \pi_k(M)$ for $1<k<2n-1$.
      
      \end{proof}  

\begin{rem}
    Note that in the above proof, if $\pi_1(M)=\pi_1(W)$ is finite, we have $H_k(\tilde{W},\partial \tilde{W})\cong H_k(D^{2n},S^{2n-1})$ and $\pi_1(\tilde{W})=0$ implies by h-cobordism theorem that $\tilde{W}$ is diffeomorphic to a ball \cite{kwon2024dynamically} and $\partial\tilde{W}$ is homeomorphic to a sphere. Therefore such dynamically contact manifolds are necessarily covered by spheres which are the only examples known to the author.
\end{rem}

\section{Unit Cotangent bundles and dynamical convexity}\label{cot}
In this section, we prove theorem \ref{thm 1.11}, which states that unit cotangent bundles cannot be strongly dynamically convex. The main idea of the proof is to use Viterbo's isomorphism \ref{Viterbo} and dynamical convexity to obtain topological information on the free loop space $\mathcal{L}M$ of the base manifold $M$. More precisely, using Viterbo isomorphism, dynamical convexity would imply that  $ \iota_* :  H_k(M,\mathbb{Z}/2\mathbb{Z})\rightarrow H_k(\mathcal{L}M,\mathbb{Z}/2\mathbb{Z})$ induced by the inclusion of constant loops, is an isomorphism for $k\leq n-1$. We can then use the Serre spectral sequence for the loop-loop fibration and the path-loop fibration to find a contradiction from this fact. A key point in the proof is that the spectral sequence retains information about the  map being an inclusion.

The first result implies that we can use the Leray-Serre spectral sequence without local coefficients.
\begin{thm}\label{pi_1}
    If the unit cotangent bundle $ST^*M$ of a smooth manifold $M$ is strongly dynamically convex , then $\pi_1(M)=0.$ 
\end{thm}

\begin{proof}
    Let $ST^*M $ be strongly dynamically convex and assume that $\pi_1(M)\neq 0$. Let $\hat{M}$ be a Liouville filling of $ST^*M$ such that the Liouville form restricts to a dynamically convex contact form on $ST^*M$ and $\hat{M}$ is Liouville isomorphic to $DT^*M$ (see remark \ref{contact form}). Then the invariance of positive symplectic homology for Liouville isomorphic manifolds implies that $SH^+(\hat{M})\cong SH^+(T^*M)$.Strong dynamical convexity implies  that $SH_k^+(T^*M)=0$ for $k\leq n$ (\ref{orb}). Since $M$ is not simply connected, $\mathcal{L}(M)$ has more than one connected component, so $H_0(\mathcal{L}M,M )\neq 0$.
    Now using the positive Viterbo isomorphism theorem \ref{pvit}, $SH_k^+(T^*M)\cong H_k(\mathcal{L}M,  M )$ we get that $ SH_0 ^+ (T ^* M)\neq 0 $. This would imply the existence of degree $0$ Reeb orbits and therefore a contradiction.
\end{proof}

The next lemma is the backbone for the contradiction we are going to derive. We use $\mathbb{Z}_2$ coefficients so that we can use the Viterbo isomorphism for all closed manifolds irrespective of whether they are not spin or orientable.

\begin{lem}\label{isomorphism}
    If the unit cotangent bundle $ST^*M $ of a simply connected manifold $M$ admits a non-degenerate strongly dynamically convex contact form $\lambda$, then the map $ \iota_* :  H_k(M;\mathbb{Z}_2)\rightarrow H_k(\mathcal{L}M;\mathbb{Z}_2)$ induced by the inclusion of constant loops, is an isomorphism for $k\leq n-1$.
\end{lem}
    
    \begin{proof}
    We consider a Liouville filling $(\widehat{M},\widehat{\lambda})$ of $ST^*M$ such that $\widehat{\lambda}\vert_{ST^*M}=\lambda$ and $(\widehat{M},\widehat{\lambda})$ is Liouville isomorphic to the disc cotangent bundle. Now since positive symplectic homology of dynamically convex contact manifolds is an invariant of Liouville manifolds, we get that  $SH^+(DT^*M)\cong SH^+(\widehat{M})$.      
    
    The long exact sequence from proposition \ref{les} in the case of the disc cotangent bundle $DT^*M$ is
        \begin{alignat*}{2}
\cdots & \to SH^+_{k+1}(DT^*M;\mathbb{Z}_2)\to H_{k+n}(DT^*M,ST^*M;\mathbb{Z}_2 )& \to SH_{k}(DT^*M;\mathbb{Z}_2)\to 
        SH^+_{k}(DT^*M;\mathbb{Z}_2)\to\cdots         
\end{alignat*}   
  Now, $ST^*M$ being dynamically convex implies that the groups $SH^+_k(D^*TM;\mathbb{Z}_2)$ vanish for $k\leq n$. Therefore, the exact sequence  gives $$H_{k+n}(DT^*M,ST^*M;\mathbb{Z}_2) \cong SH_{k}(DT^*M,\mathbb{Z}_2)\quad\textrm{for}\quad k\leq n-1.$$ By Thom isomorphism we have $$H_{k+n}(DT^*M,ST^*M;\mathbb{Z}_2) \cong H_k(M;\mathbb{Z}_2)$$ and Viterbo's isomorphism implies that $$SH_{k}(DT^*M,\mathbb{Z}_2) \cong  H_k(\mathcal{L}M,\mathbb{Z}_2).$$ Under these identifications, the map $H_{k+n}(DT^*M,ST^*M;\mathbb{Z}_2) \to SH_{k}(DT^*M;\mathbb{Z}_2)$ is the map induced by inclusion of constant loops
  $\iota_* :  H_k(M;\mathbb{Z}_2)\rightarrow H_k(\mathcal{L}M;\mathbb{Z}_2)$ which is an isomorphism in the case $k\leq n-1$
  \end{proof}

To derive the contradiction we work with the path-loop fibration and loop-loop fibration to compute the homology groups $H_k(\Omega M;\mathbb{Z}_2)$, $ H_k(\mathcal{L}M;\mathbb{Z}_2)$ and $H_k(M;\mathbb{Z}_2)$ using the above lemma. We recall the definitions of these fibrations, the Leray-Serre exact sequence in homology and several of its properties which we shall use.

The path space $\mathcal{P}_xX$ of a manifold $X$ is given by  $\mathcal{P}_xX=\{\gamma:[0,1]\rightarrow X, \gamma(0)=x\}  $.

\begin{thm}\label{path}
     The evaluation map $p:\mathcal{P}_xX\rightarrow X$ given by $p(\gamma)=\gamma(1)$ is a Serre fibration with fibre the based loop space $\Omega_xX=\{\gamma:[0,1]\rightarrow X, \gamma(0)=\gamma(1)=x\}$
\end{thm}
\begin{thm}\label{loop}
    The evaluation map $p:\mathcal{L}X\rightarrow X$ given by $p(\gamma)=\gamma(0)$ is a Serre fibration with fibre $\Omega_x(X)$. Furthermore $p\circ i=Id$ where $i:X\rightarrow \mathcal{L}X$ is the inclusion of constant loops.
\end{thm}

\begin{rem}
    The fibration in \ref{path} is called the \textbf{path-loop fibration} and the fibration in \ref{loop} is called the \textbf{loop-loop fibration}
\end{rem}
\begin{thm}[\cite{serre1951homologie}\cite{leray1950anneau}]\textbf{Leray-Serre Spectral Sequence} Let $F\xrightarrow {i}E\xrightarrow{p}B $ be a fibration with $B$ simply connected. Then, there exists a collection of differential bi-graded abelian groups $\{E_{p,q}^r, d_r\}$ such that 
\begin{enumerate}
    \item $d_r: E_{p,q}^r\rightarrow E_{p-r,q+r-1}$ satisfies $d_r\circ d_r=0$ and $E_{p,q}^{r+1}=ker d_r/Im d_r$
    \item $E^r_{p,q} $ converges to $ H_{p+q}(E)$ 
\end{enumerate}

\end{thm}
\begin{cor} [\cite{mccleary2001user}Example $5$D]
\textbf{Serre exact sequence}\label{Serre ex}: Let $F\xrightarrow {i}E\xrightarrow{p}B $ be a fibration with $B$ simply connected. Suppose $H_i(B)=0$ for $0<i<p$ and $H_j{(F)=0}$ for $0<j<q$. Then there exist an exact sequence   
     
     $ \vspace{0.2cm}
      \hspace{2cm} H_{p+q-1}(F)\rightarrow H_{p+q-1}(E)\rightarrow H_{p+q-1}(B)\rightarrow H_{p+q-2}(F)\rightarrow \cdots H_1(E)\rightarrow 0 $
     \end{cor}

\begin{prop}\label{prop:HkM}
    Let $M$ be a simply connected manifold and let $ \iota^* :  H_k(M;\mathbb{Z}_2)\rightarrow H_k(\mathcal{L}M;\mathbb{Z}_2)$ be the map induced by inclusion of constant loops. If $\iota_*$  is an isomorphism for $k \leq n-1$, then
    $$H_k(M;\mathbb{Z}_2)=0\quad\textrm{for}\quad1\leq k\leq n-1,$$ $$H_k(\mathcal{L}M;\mathbb{Z}_2)=0\quad\textrm{for}\quad1\leq k\leq n-1,$$  and,
    $$H_k(\Omega M;\mathbb{Z}_2)=\begin{cases}
        \mathbb{Z}_2 & \textrm{for } k=n-1\\
        0 & \textrm{for } 1 \leq k\leq n-2.
    \end{cases}$$
\end{prop}
\begin{proof}
 Note that $\pi_1(M)=0$ implies $H_1(M)=0$ and therefore $H_1(\mathcal{L} M)=0$. Let $\pi:\mathcal{L}M\rightarrow M $ be the fibration map, then $\pi\circ\iota=Id \implies  \pi_*\circ\iota_*=Id$. Therefore, $\iota_* $ being an isomorphism for $ k\leq n-1$  implies that $\pi_*$ is an isomorphism for $k\leq n-1$.
 
 Throughout the proof, we work with $\mathbb{Z}_2$ coefficients which we omit from the notations for the sake of brevity. This choice of coefficients also resolves any problems regarding extensions in the limiting page.
 We use the Serre spectral sequences for the path-loop fibration and the loop-loop fibration as well as an induction to prove our claim.

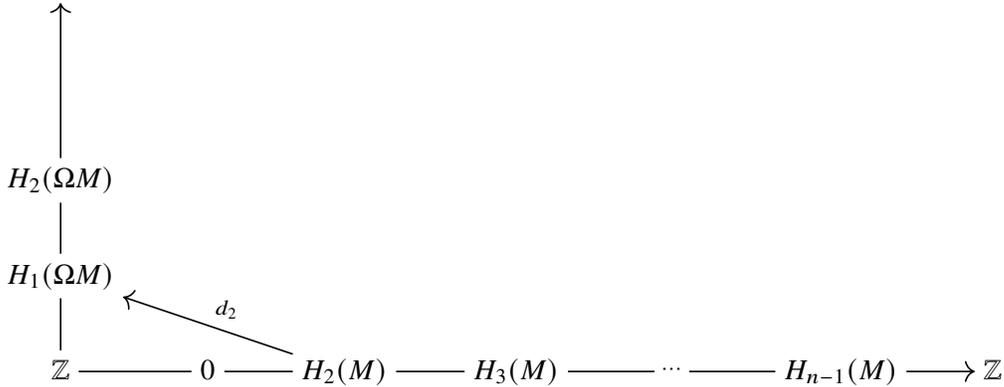
\begin{figure}[!htbp]
    \centering

\[\begin{tikzcd}[cramped]
	{} \\
	\\
	\\
	{H_2({\Omega M})} \\
	{H_1(\Omega M)} \\
	{\mathbb{Z}} & 0 & {H_2(M)} & {H_3(M)} &&& {H_{n-1}(M)} & {\mathbb{Z}} \\
	&&&&&&&& {}
	\arrow[from=4-1, to=1-1]
	\arrow[no head, from=5-1, to=4-1]
	\arrow[no head, from=6-1, to=5-1]
	\arrow[no head, from=6-1, to=6-2]
	\arrow[no head, from=6-2, to=6-3]
	\arrow["{d_2}"', from=6-3, to=5-1]
	\arrow[no head, from=6-3, to=6-4]
	\arrow["\cdots"{description}, no head, from=6-4, to=6-7]
	\arrow[from=6-7, to=6-8]
\end{tikzcd}\] 

\caption{Page $2$ of the spectral sequence for the loop-loop fibration}
    
\end{figure}

Since the edge homomorphism $H_2(\mathcal{L}M)\rightarrow E^\infty_{(2,0)}\rightarrow \cdots \rightarrow E^2_{(2,0)}=H_2(M)$ in the spectral sequence for the loop-loop fibration coincides with the fibration map $\pi_*$ \cite[Proposition 64.2]{miller2021lectures} , we have that $E^2_{(2,0)}=H_2(M)$ persists till the limiting page. This implies that $d_2=0$. The fact that the differential $d_2$ vanishes forces $E^2_{(0,1)}=H_1(\Omega M )$ to persist until the last page. The isomorphism $H_2(\mathcal{L}M)=H_2(M)$ implies that in the limiting page $H_2(M)$ is the sole contributor to $H_2(\mathcal{L}M)$. Therefore, $H_1(\Omega M )=0$.

We have the Serre exact sequence \ref{Serre ex} in low degrees for the path-loop fibration,
$$H_2(\mathcal{P}M)\rightarrow H_2(M)\rightarrow H_1(\Omega M)\rightarrow H_1(\mathcal{P}M)\rightarrow H_1(M)\rightarrow 0$$

 The contractibility of the path space $\mathcal{P}M$ implies $H_2(M)\cong H_1(\Omega M)$ and therefore $$H_2(M)= H_2(\mathcal{L}M)=0$$

Now, assume by induction, that for $ 0<i<p<n$,  $H_i(M)=0$ and $H_{i-1}(\Omega M)=0$. 
The edge homomorphism $H_p(\mathcal{L}M)\rightarrow E^\infty_{(p,0)}\rightarrow \cdots \rightarrow E^2_{(p,0)}=H_p(M)$ coincides with the fibration map $\pi_*$ and this implies that $ E^2_{(p,0)}=H_p(M)$ would remain until the limiting page. Therefore $d_p=0$. This forces $H_{p-1}(\Omega M)$ to persist until the limiting page.

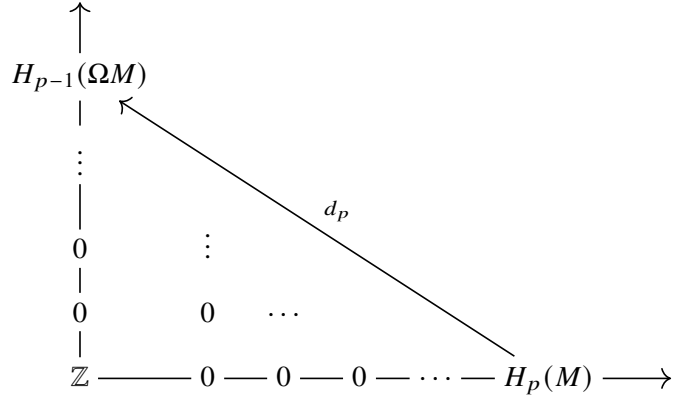
\begin{figure}[!hbtp]
    \centering

\[\begin{tikzcd}[cramped,sep=small]
	{} \\
	\\
	{H_{p-1}(\Omega M)} \\
	\vdots \\
	0 & \vdots \\
	0 & 0 & \cdots \\
	{\mathbb{Z}} & 0 & 0 & 0 & \cdots & {H_p(M)} && 
	\arrow[from=3-1, to=1-1]
	\arrow[no head, from=4-1, to=3-1]
	\arrow[no head, from=5-1, to=4-1]
	\arrow[no head, from=6-1, to=5-1]
	\arrow[no head, from=7-1, to=6-1]
	\arrow[no head, from=7-1, to=7-2]
	\arrow[no head, from=7-2, to=7-3]
	\arrow[no head, from=7-3, to=7-4]
	\arrow[no head, from=7-4, to=7-5]
	\arrow[no head, from=7-5, to=7-6]
	\arrow["{d_p}"', from=7-6, to=3-1]
	\arrow[from=7-6, to=7-8]
\end{tikzcd}\]

\caption{$p$th page of the spectral sequence for the loop-loop fibration}
    
\end{figure}

The fact that $H_{p-1}(\Omega M)$ persists and  that $H_{p-1}(\mathcal{L}M)=\bigoplus E^\infty_{(i,p-1-i)}$ implies $H_{p-1}(\Omega M)=0.$

Using the Serre exact sequence for path-loop fibration, we have 

$ \hspace{0.5cm} \rightarrow H_{2p-2}(\Omega M)\rightarrow \cdots\rightarrow H_{p}(\Omega M)\rightarrow H_{p}(\mathcal{P}M)\rightarrow H_{p}(M)\rightarrow H_{p-1}(\Omega M)\rightarrow 0\rightarrow\cdots $

The contractibility of $\mathcal{P}M$ implies that $H_p(M)\cong H_{p-1}(\Omega M)=0$.
Therefore $H_p(\mathcal{L}M)\cong H_p(M)=0$.

Finally, we have to show $H_{n-1}(\Omega M)=\mathbb{Z}_2$ but this follows directly from the Serre exact sequence of the path-loop fibration and the fact that the path space is contractible.
\end{proof}

  \begin{rem}
      To prove proposition \ref{prop:HkM}, we used the extra information that the isomorphism in \ref{isomorphism} is induced by the inclusion of constant loops.
    \end{rem}

\begin{thm}\label{contradiction}
    The unit cotangent bundle $ST^*M$ of a simply connected closed manifold $M$ (not necessarily orientable) with standard contact structure cannot be strongly dynamically convex.
\end{thm}
\begin{proof}
We examine the Serre exact sequence \ref{Serre ex} of the loop-loop fibration
\begin{equation}
    H_n(\Omega M)\rightarrow H_n(\mathcal{L}M)\xrightarrow{p_*} H_n(M)\xrightarrow{\delta} H_{n-1}(\Omega M)\xrightarrow{i_*} H_{n-1}(\mathcal{L}M)\rightarrow 0
\end{equation}
Note that $p\circ i_M =Id$ (where ${i}_M$ is inclusion of constant loops) implies $p_*\circ {i_M}_*=id$ and therefore $p_*$ is surjective. By exactness, we have $\delta=0$ and hence $i_*: H_{n-1}({\Omega M})\rightarrow H_{n-1}(\mathcal{L}M)$  is injective. Then, $H_{n-1}(\Omega M)=\mathbb{Z}_2$ and $H_{n-1}(\mathcal{L}M)=0$, gives us a contradiction.  
\end{proof}

\begin{cor}[Theorem \ref{thm 1.11}]\label{sdc}
    The unit cotangent bundle $ST^*M$ of a closed manifold $M$ cannot be strongly dynamically convex.
    
\end{cor}
\begin{proof}
    Theorem \ref{pi_1} implies $M$ is simply connected, we obtain the result, using theorem \ref{contradiction}.
\end{proof}
\begin{cor}[Theorem  \ref{contra}]\label{change}
    The unit cotangent bundle $ST^*M$ of a closed manifold $M$ cannot be dynamically convex if it is simply connected.
\end{cor}

\section{Dynamical convexity and surgery}\label{5}
The idea to obstruct dynamical convexity after surgery on Weinstein fillable contact manifolds is to observe the changes in $SH^+$ and in the homology after performing surgery.
We refer to \cite{fauck2020manifolds} for details on how attaching a handle induces maps on the canonical long exact sequence of symplectic homology. In particular, the author shows that, after handle attachment, the Viterbo transfer map \ref{vit} still respects the action filtration and therefore induces a map on the long exact sequence \ref{les}.

Note that $SH=\bigoplus_{\gamma\in \pi_1} SH_\gamma$ where $SH_\gamma$ denotes the subgroup generated by orbits in the homotopy class of $\gamma$. In this section we only consider the subgroup generated by contractible orbits and denote it by $SH$ or $SH^+$ for the sake of brevity. We consider $SH,SH^+$ and homology/cohomology groups over $\mathbb{Q}$ coefficients but omit the coefficients from the notation.

If $W^\prime$ is the domain obtained from attaching an index-$k$ handle to a Weinstein domain $W$ we have that $$H^*(W^\prime,W)\cong H^*(D^k,S^{k-1}).$$ Therefore the long exact sequence in cohomology

\[\begin{tikzcd}[cramped]
	\cdots & {H^*(W^\prime,W)} & {H^*(W^\prime)} & {H^*(W)} & {H^{*+1}(W^\prime,W)} & \cdots
	\arrow[from=1-1, to=1-2]
	\arrow[from=1-2, to=1-3]
	\arrow[from=1-3, to=1-4]
	\arrow[from=1-4, to=1-5]
	\arrow[from=1-5, to=1-6]
\end{tikzcd}\]

 implies that $H^n(W^\prime)\cong H^{n}(W)$ when $n \neq k,k-1.$ The following part of the long exact sequence 

\[\begin{tikzcd}[cramped]\label{ex}
	0 & {H^{k-1}(W^\prime)} & {H^{k-1}(W)} & {\mathbb{Z}} & {H^{k}(W^\prime)} & {H^{k}(W)} & 0
	\arrow[from=1-1, to=1-2]
	\arrow[from=1-2, to=1-3]
	\arrow[from=1-3, to=1-4]
	\arrow[from=1-4, to=1-5]
	\arrow[from=1-5, to=1-6]
	\arrow[from=1-6, to=1-7]
\end{tikzcd}\]

 implies that 
 \[
    \textrm{either}\quad\begin{cases}
        \operatorname{rk}(H^{k-1}(W) =\operatorname{rk}(H^{k-1}(W^\prime))+1\\
        \operatorname{rk}(H^{k}(W^\prime))=\operatorname{rk}(H^{k}(W))
    \end{cases}
    \qquad\textrm{or}\quad
    \begin{cases}
        \operatorname{rk}(H^{k-1}(W^\prime))= \operatorname{rk}(H^{k-1}(W)\\
        \operatorname{rk}(H^{k}(W^\prime))=\operatorname{rk}(H^{k}(W))+1
    \end{cases}
\]

\begin{rem}
    The following proofs in this section also hold for non-Weinstein exact fillable contact manifolds if the fillings  have vanishing rational first chern class.
\end{rem}

\begin{thm}[Theorem \ref{subcritical}]\label{surg}
    Let $M$ be a dynamically convex Weinstein fillable contact manifold of dimension greater than $5$. Let $M^\prime$ be a manifold that is obtained by subcritical or flexible surgery on $M$, then $M^\prime$ cannot be dynamically convex.
\end{thm}

\begin{proof}
    Assume $M^\prime$ is  dynamically convex. Let $W$ and $W^\prime$ be the Weinstein fillings of $M$ and $M^\prime$ respectively. By propositon [\ref{chern}], $c_1(M,\mathbb{Q})=0$ and $c_1(M^\prime,\mathbb{Q})=0$ implies that $c_1(W,\mathbb{Q})=0$ and $c_1(W^\prime,\mathbb{Q})=0$ respectively. Therefore, we have a rational grading for the symplectic homology groups.
    
    The maps induced by surgery on the canonical long exact sequence \cite[section 2.8]{fauck2020manifolds} are
\[\small{\begin{tikzcd}[cramped,column sep=small,row sep=scriptsize]
	\cdots & {SH^+_{n-k+2}(W^\prime)} & {H^{k-1}(W^\prime)} & {SH_{n-k+1}(W^\prime)} & {SH^+_{n-k+1}(W^\prime)} & {H^k(W^\prime)} & {SH_{n-k}(W^\prime)} & {SH^+_{n-k}(W^\prime)} \\
	\cdots & {SH^+_{n-k+2}(W)} & {H^{k-1}(W)} & {SH_{n-k+1}(W)} & {SH^+_{n-k+1}(W)} & {H^k(W)} & {SH_{n-k}(W)} & {SH^+_{n-k}(W)}
	\arrow[from=1-1, to=1-2]
	\arrow[from=1-2, to=1-3]
	\arrow[from=1-2, to=2-2]
	\arrow[from=1-3, to=1-4]
	\arrow[from=1-3, to=2-3]
	\arrow[from=1-4, to=1-5]
	\arrow[from=1-4, to=2-4]
	\arrow[from=1-5, to=1-6]
	\arrow[from=1-5, to=2-5]
	\arrow[from=1-6, to=1-7]
	\arrow[from=1-6, to=2-6]
	\arrow[from=1-7, to=1-8]
	\arrow[from=1-7, to=2-7]
	\arrow[from=1-8, to=2-8]
	\arrow[from=2-1, to=2-2]
	\arrow[from=2-2, to=2-3]
	\arrow[from=2-3, to=2-4]
	\arrow[from=2-4, to=2-5]
	\arrow[from=2-5, to=2-6]
	\arrow[from=2-6, to=2-7]
	\arrow[from=2-7, to=2-8]
\end{tikzcd}}\]
In the above diagram, the induced maps on $SH$ are isomorphisms as subcritical handle attachment preserves symplectic homology. By dynamical convexity, the positive symplectic homology groups $SH^+$ are $0$ for $W$ and $W^\prime$ in degree less than $n+1$.

Assume the index $k$ of the handle is greater than one, then the $5$ lemma implies that the homology groups at level $k$ and $k-1$ are isomorphic. This is a contradiction since attaching a handle changes the rank of one of the above groups by $1$. If the handle is of index $1$, $H^0(W)\cong H^0(W^\prime)$, since connectedness is preserved. Therefore the ranks of $H^1(W^\prime)$ and $H^1(W)$  differ by $1$ but the $5$ lemma implies they are isomorphic. Hence a contradiction.
\end{proof}
\begin{cor}
Let $M$ be a Weinstein fillable contact manifold of dimension greater than $5$. Assume $M$ is strongly dynamically convex. Let $M^\prime$ be a manifold obtained from $M$ by a subcritical or a flexible surgery, then $M^\prime$ cannot be strongly dynamically convex.
\end{cor}
\begin{proof}
    Since strong dynamical convexity implies dynamical convexity, \ref{surg} implies the corollary.
\end{proof}

\begin{thm}[Theorem \ref{contact 0}]\label{con 0}
    Let $M^\prime $ be a manifold obtained by a contact $0-$surgery on a Weinstein -fillable contact manifold with vanishing rational first chern class. Then $M^\prime$ cannot be dynamically convex.
\end{thm}
\begin{proof}
    Let $W$ and $W^\prime$ be the Weinstein fillings of $M$ and $M^\prime$ respectively. Assume that $M^\prime$ is dynamically convex. We have by \ref{chern}, that the chern classes of the fillings vanish. Dynamical convexity of $M^\prime$, implies that $SH_k^+(W^\prime)=0$ for $k<n+1$. Now, observing the diagram induced by surgery on the long exact sequence,

\[\begin{tikzcd}[cramped]
	{H^0(W^\prime)} & {SH_n(W^\prime)} & {SH^+_n(W^\prime)} & {H^1(W^\prime)} & {SH_{n-1}(W^\prime)} \\
	{H^0(W)} & {SH_n(W)} & {SH^+_n(W)} & {H^1(W)} & {SH_{n-1}(W)}
	\arrow[from=1-1, to=1-2]
	\arrow[from=1-1, to=2-1]
	\arrow[from=1-2, to=1-3]
	\arrow[from=1-2, to=2-2]
	\arrow[from=1-3, to=1-4]
	\arrow[from=1-3, to=2-3]
	\arrow[from=1-4, to=1-5]
	\arrow[from=1-4, to=2-4]
	\arrow[from=1-5, to=2-5]
	\arrow[from=2-1, to=2-2]
	\arrow[from=2-2, to=2-3]
	\arrow[from=2-3, to=2-4]
	\arrow[from=2-4, to=2-5]
\end{tikzcd}\]
  we note that, since attaching a 1-handle is subcritical, the second and last arrows on symplectic homology are isomorphisms. Also note that since connectedness is preserved $H^0$ remains invariant and therefore $\operatorname{rk}(H^1(W^\prime))=\operatorname{rk}(H^1(W))+1$. The map $H^1(W^\prime)\rightarrow H^1(W)$, induced by inclusion, is surjective by the long exact sequence in cohomology in section \ref{ex}. Therefore, by the 4-lemma (\cite[Chapter 12 Lemma 3.1]{maclane2012homology}), the map $SH^+_n(W^\prime)\rightarrow SH^+_n(W)$ is a surjection. The vanishing of $SH^+_n(W^\prime)$ implies that $SH^+_n(W)$ vanishes.
  
  Since the second, third and fifth vertical maps are isomorphisms, the $4-$lemma implies that the fourth vertical map is injective which is a contradiction.
\end{proof}

\begin{cor}
     Let $M^\prime $ be a manifold obtained by performing $0-$surgery on a Weinstein -fillable contact manifold with vanishing rational chern class. Then $M^\prime$ cannot be strongly dynamically convex.
\end{cor}

\begin{thm}
    Let $M^{2n-1}$, $n\geq 2$,  be a dynamically convex manifold which admits a Liouville filling $W^{2n}$. Then either $W$ is a homology sphere  or $SH_n(W)=\mathbb{Z}$.
\end{thm}
\begin{proof}
    We examine part of the canonical long exact sequence,
    $$H^0(W)\rightarrow SH_n(W)\rightarrow SH^+_{n}(W).$$
    Since $M:=\partial W$ is dynamically convex, $SH^+_{n}(W)=0.$ This implies that the map $H^0(W)\rightarrow SH_n(W)$ is surjective. Note that if we work with coefficients over any field $k$, we have that $SH_n(W)=k $ or $0$. Using the universal coefficients theorem, we have $SH_n(W)=\mathbb{Z}$ or $0$. If $SH_n(W)=0$, then $SH(W)=0$, since the unit is in $SH_n(W)$. In fact, $SH(W)=0$ and $SH^+_k(W)=0$ for $k<n+1$ implies that $H_{q}(W,\partial W)=0$ for $q<2n$. Using the fact that W is compact with boundary $\partial W$, we have $H_{2n}(W,\partial W)=\mathbb{Z}$.
\end{proof}
\begin{cor}\label{diff}
    If $M$ is a simply connected fillable contact manifold of dimension at least $5$, then $M$ is diffeomorphic to a sphere or $SH_n(W)=\mathbb{Z}$ for any Liouville filling. 
\end{cor}
\begin{proof}
    If $M$ is simply connected, dynamically convex and of dimension at least $5$, then by \cite{kwon2024dynamically} $M$ is diffeomorphic to a sphere if the filling has vanishing $SH$. This  follows from the h-cobordism theorem. Note that vanishing of $SH^+$ is independent of the filling by \ref{rem}.
\end{proof}
\begin{rem}
    Note that a simply connected dynamically convex contact manifold is strongly dynamically convex, so all such manifolds that are non-diffeomorphic to spheres should have $SH_n(W)=\mathbb{Z}$ for any filling if such manifolds exist.
\end{rem}
    
\begin{cor}[Theorem \ref{SH}]\label{SH_n}
    Let $ \Lambda$ be the class of manifolds that admit fillings with non-vanishing symplectic homology, then for $Y_1,\ldots Y_k\in\Lambda$, $k>1$, the connected sum $\#_kY_i$ and the manifolds obtained by subcritical/flexible surgeries on this class of connected sums cannot be dynamically convex. 
\end{cor}
\begin{proof}
    Let $K:= \#_kY_i$. For a filling $W$ of $K$,  $SH(W_1\#W_2)=SH(W_1)\oplus SH(W_2)$. This implies that $\operatorname{rk}(SH_n(K))>1$. By \ref{diff}, $K$ is not dynamically convex. Let $K^\prime$ be the manifold obtained by subcritical surgery on $K$. Then since subcritical/flexible surgery induces a isomorphism on symplectic homology groups, \ref{diff} implies that $K^\prime$ cannot be dynamically convex. 
\end{proof}

\begin{thm}[Theorem \ref{hom}]\label{hom1}
    Let $W^{2n}$, $n\geq 3$, be a Weinstein domain which is a rational homology ball, that is $H_k(W,\partial W)=\mathbb{Q}$ for $k=2n$ and $0$ otherwise. Then attaching a flexible/subcritical handle produces a manifold $W^\prime$ such that $\partial W^\prime$ cannot be dynamically convex.
\end{thm}
\begin{proof}
Assume $\partial W^\prime$ is  dynamically convex. Since $W$ is a rational homology ball, $H^2(W)=0$ by Poincare-Lefshetz duality which implies $c_1(\partial W,\mathbb{Q})=0.$ Therefore by \ref{chern}, the Chern classes vanish. Let $k$ be the index of the handle. We examine the maps induced by surgery on the long exact sequence,

\[\begin{tikzcd}[cramped,column sep=small,row sep=scriptsize]
	 & {H^{k-1}(W^\prime)} & {SH_{n-k+1}(W^\prime)} & {SH^+_{n-k+1}(W^\prime)} & {H^k(W^\prime)} & {SH_{n-k}(W^\prime)}  \\
	 & {H^{k-1}(W)} & {SH_{n-k+1}(W)} & {SH^+_{n-k+1}(W)} & {H^k(W)} & {SH_{n-k}(W)} 
	\arrow[from=1-2, to=1-3]
	\arrow[from=1-2, to=2-2]
	\arrow[from=1-3, to=1-4]
	\arrow[from=1-3, to=2-3]
	\arrow[from=1-4, to=1-5]
	\arrow[from=1-4, to=2-4]
	\arrow[from=1-5, to=1-6]
	\arrow[from=1-5, to=2-5]
	\arrow[from=1-6, to=2-6]
   \arrow[from=2-2, to=2-3]
	\arrow[from=2-3, to=2-4]
	\arrow[from=2-4, to=2-5]
	\arrow[from=2-5, to=2-6]
\end{tikzcd}\]
Since the index $k$ of the handle is less than $n+1$ and since subcritical/flexible surgery preserves $SH$, we have that the second and the fifth vertical  maps are isomorphisms. As before, we have
\[
\textrm{either,}\quad
\begin{cases}\operatorname{rk}(H^{k-1}(W)) =\operatorname{rk}(H^{k-1}(W^\prime))+1 \textrm{ and}\\\operatorname{rk}(H^{k}(W^\prime))=\operatorname{rk}(H^{k}(W))\end{cases}
\qquad\textrm{or,}\quad
\begin{cases}
    \operatorname{rk}(H^{k-1}(W^\prime))= \operatorname{rk}(H^{k-1}(W) \textrm{ and}\\\operatorname{rk}(H^{k}(W^\prime))=\operatorname{rk}(H^{k}(W))+1
\end{cases}.
\]
Since $W$ is a rational homology ball, we cannot have $\operatorname{rk}(H^{k-1}(W)) =\operatorname{rk}(H^{k-1}(W^\prime))+1$. Therefore $\operatorname{rk}(H^{k-1}(W^\prime))= \operatorname{rk}(H^{k-1}(W)$ and $\operatorname{rk}(H^{k}(W^\prime))=\operatorname{rk}(H^{k}(W))+1$.

Since the map $H^k(W,\mathbb{Q})\rightarrow H^k(W,\mathbb{Q}) $ is induced by inclusion, it is surjective. Therefore using the $4-$lemma, we get that the map $SH^+_{n-k+1}(W^\prime)\rightarrow SH^+_{n-k+1}(W)$ is surjective. Since the  maps 
$SH^+_{n-p+1}(W^\prime)\rightarrow SH^+_{n-p+1}(W)$  are epimorphisms, using the 4-lemma on the surgery diagram, we obtain that $$\operatorname{rk}(SH^+_q(W^\prime))\geq \operatorname{rk}(SH^+_q(W))$$ for all $q\in \mathbb{Z}$; in particular for $q\leq n$. If $(SH^+_q(W))\neq 0$ for some $q\leq n$, then $SH^+_q(W^\prime)\neq 0$ and therefore $M^\prime$ cannot be dynamically convex. Therefore we must have $SH^+_q(W)=SH^+_q(W^\prime)=0$ for $q\leq n$. The long exact sequence \ref{les} for $W$  implies $SH_q(W)\cong H^{n-q}(W)=0$ for $q< n.$ Since subcritical/flexible surgery doesn't change $SH$, we have $SH_q(W^\prime)= SH_q(W)=0$ for $q<n$. The long exact sequence \ref{les} for $W^\prime$ thus implies $H^i(W^\prime)=0 $ for $i\neq 0$ and $H^0(W^\prime)=0.$ This is a contradiction since $H^k(W^\prime)\neq 0.$
\end{proof}
\begin{cor}
    Let $M^\prime$ be a contact manifold obtained by subcritical/flexible surgery on a contact manifold $M$ admitting a Weinstein filling which is a rational homology ball. Then $M^\prime$ cannot be strongly dynamically convex.
\end{cor}

\begin{thm}[Theorem \ref{cotangent connect}] \label{cot con}
    The connected sum of any Weinstein fillable contact manifold with any unit cotangent bundle $DT^*N$ of a simply connected manifold $N$ cannot be dynamically convex.
\end{thm}

\begin{proof}
Let $M$ denote the disjoint union of the unit cotangent bundle and the Weinstein fillable manifold and  $M^\prime$ denote their connected sum. Let $W$ be the filling of $M$ and $W^\prime$ the Weinstein filling of $M^\prime$. Assume that $M^\prime$ is  dynamically convex. Using proposition \ref{chern}, the rational first Chern classes of the fillings vanish.

We  consider the process of connected sum as attaching a one handle to two components of $M$, thus decreasing the rank of $H^0$ by 1 and preserving the ranks of all the higher cohomology groups. $SH$ remains invariant as we are performing subcritical surgery. 

Now, if $SH^+_k(W)=0 $ for $k<n+1$, then $SH^+_k(DT^*N)=0$ for $k<n+1$. The proof of lemma \ref{isomorphism} holds and this implies $H_k(N,\mathbb{Z}/2\mathbb{Z})\hookrightarrow H_k(\mathcal{L}N,\mathbb{Z}/2\mathbb{Z})$ under inclusion is an isomorphism for $k< n+1$. Since $N$ is simply connected the argument following Lemma \ref{isomorphism} in section \ref{cot} using the Serre spectral sequence would give a contradiction. Therefore there exists $k< n+1$ such that $SH^+_k(W)=0 $.

Therefore by analyzing the same induced sequence by surgery as above and using the five lemma we get that if $SH^+_k(W)\neq 0$ for $k<n$, then $SH^+_k(W^\prime)\neq 0$ and therefore the connected sum is not  dynamically convex. When $SH^+_n(W)\neq0$, we consider a part of the diagram of exact sequences

\[\begin{tikzcd}[cramped]
	{SH_n(W^\prime)} & {SH^+_n(W^\prime)} & {H^1(W^\prime)} & {SH_{n-1}(W^\prime)} \\
	{SH_n(W)} & {SH^+_n(W)} & {H^1(W)} & {SH_{n-1}(W)}
	\arrow[from=1-1, to=1-2]
	\arrow[two heads, from=1-1, to=2-1]
	\arrow[from=1-2, to=1-3]
	\arrow[from=1-2, to=2-2]
	\arrow[from=1-3, to=1-4]
	\arrow[two heads, from=1-3, to=2-3]
	\arrow[two heads, from=1-4, to=2-4]
	\arrow[from=2-1, to=2-2]
	\arrow[from=2-2, to=2-3]
	\arrow[from=2-3, to=2-4]
\end{tikzcd}\]
The first, third and fourth vertical maps are isomorphisms, therefore the $4$-lemma implies that the second map is surjective and therefore $SH^+_n(W^\prime)\neq 0$. Thus, $W^\prime$ cannot be strongly dynamically convex. 
\end{proof}
\begin{cor}
    The connected sum of a Weinstein fillable contact manifold with the unit cotangent bundle of a simply connected manifold cannot be strongly dynamically convex.
\end{cor}

\bibliographystyle{alpha}
\bibliography{ref_dynamicallyconvex}

\end{document}